\documentclass{amsart}
\usepackage[english]{babel}
\usepackage{amsfonts}
\usepackage{amsthm}
\usepackage{amssymb}
\usepackage{mathrsfs}
\usepackage{enumerate}
\input xy
\xyoption{all}

\newtheorem{teor}{Theorem}

\newtheorem{lemma}{Lemma}
\DeclareMathOperator{\PSL}{PSL}
\DeclareMathOperator{\PGL}{PGL}
\DeclareMathOperator{\ASL}{ASL}
\DeclareMathOperator{\SL}{SL}
\DeclareMathOperator{\GL}{GL}
\DeclareMathOperator{\Aut}{Aut}

\begin{document}

\title{Finite groups with $6$ or $7$ automorphism orbits}
\date{}
\author{Alex Carrazedo Dantas}
\address[Alex Carrazedo Dantas]{Universidade Tecnol\'ogica Federal do Paran\'a, Campus de Guarapuava, Paran\'a 85053-525, Brazil}
\email{alexcdan@gmail.com}
\author{Martino Garonzi}
\address[Martino Garonzi]{Departamento de Matem\'atica, Universidade de Bras%
\'{\i}lia, Campus Universit\'{a}rio Darcy Ribeiro, Bras\'{\i}lia - DF
70910-900, Brazil}
\email{mgaronzi@gmail.com}
\thanks{MG was supported by FEMAT-Funda\c{c}\~{a}o de Estudos em Ci\^{e}ncias Matem\'aticas Proc. 037/2016}
\author{Raimundo Bastos}
\address[Raimundo Bastos]{Departamento de Matem\'atica, Universidade de Bras%
\'{\i}lia, Campus Universit\'{a}rio Darcy Ribeiro, Bras\'{\i}lia - DF
70910-900, Brazil}
\email{bastos@mat.unb.br}

\begin{abstract}
Let $G$ be a group. The orbits of the natural action of $\Aut(G)$ on $G$ are called ``automorphism orbits'' of $G$, and the number of automorphism orbits of $G$ is denoted by $\omega(G)$. In this paper the finite non-solvable groups $G$ with $\omega(G) \leq 6$ are classified - this solves a problem posed by Markus Stroppel - and it is proved that there are infinitely many finite non-solvable groups $G$ with $\omega(G)=7$. Moreover it is proved that for a given number $n$ there are only finitely many finite groups $G$ without non-trivial abelian normal subgroups and such that $\omega(G) \leq n$, generalizing a result of Kohl.
\end{abstract}

\maketitle

\section{Introduction}

For $G$ a group denote by $\omega(G)$ the number of ``automorphism orbits'' of $G$, that is, the number of orbits of the natural action of $\Aut(G)$ on $G$. Two elements in the same such orbit are also called ``$\Aut(G)$-conjugate''. Observe that automorphism orbits are unions of conjugacy classes and hence they give an example of \textit{fusion} in the holomorph $\Aut(G) \ltimes G$, a well-known concept established in the literature (see for example \cite[Chapter 7]{gorenstein}). It is interesting to ask what can we say about $G$ only knowing $\omega(G)$. It is obvious that $\omega(G)=1$ if and only if $G = \{1\}$, and it is well-known that if $G$ is a finite group then $\omega(G)=2$ if and only if $G$ is elementary abelian (this is false in the infinite case: a construction due to Higman, Neumann and Neumann shows that there is an infinite non-abelian simple group $S$ with $\omega(S)=2$). Laffey and MacHale in \cite{lm} showed that the alternating group $\mbox{A}_5$ (which has $\omega(\mbox{A}_5)=4$) is the only finite non-solvable group $G$ with $\omega(G) \leq 4$, and Zhang in \cite{zhang} classified the finite groups in which any two elements of the same order lie in the same automorphism orbit. Markus Stroppel in \cite{stroppel2} showed that the only finite non-abelian simple groups $G$ with $\omega(G) \leq 5$ are the groups $\PSL(2,\mathbb{F}_q)$ with $q \in \{4,7,8,9\}$, and in \cite{stroppel} and \cite{stroppel2} he investigated the properties of the locally compact topological groups $G$ such that $\omega(G)$ (the number of topological automorphism orbits) is finite.

\ 

In \cite{stroppel} Stroppel proposed the following problem (Problem 2.5): classify all the finite non-solvable groups $G$ with $\omega(G) \leq 6$. In this paper we solve Stroppel's problem showing, in Section \ref{o6}, that if $G$ is a finite non-solvable group with $\omega(G) \leq 6$ then $G$ is isomorphic to one of $\PSL(2,\mathbb{F}_q)$ with $q \in \{4,7,8,9\}$, $\PSL(3,\mathbb{F}_4)$ or $\ASL(2,\mathbb{F}_4)$. Here $\ASL(2,\mathbb{F}_4)$ is the affine group $\SL(2,\mathbb{F}_4) \ltimes \mathbb{F}_4^2$ where $\SL(2,\mathbb{F}_4)$ acts naturally on $\mathbb{F}_4^2$. In this proof the only result we use that depends on the classification of finite simple groups is Kohl's classification of simple groups with at most six automorphism orbits \cite{kohl}.

\ 

By a theorem of Landau \cite{landau}, for a given number $n$ there are only finitely many finite groups with at most $n$ conjugacy classes. The case of automorphism orbits is different, for instance elementary abelian groups all have only two automorphism orbits. However by a result of Kohl \cite{kohl}, for a given number $n$ there are only finitely many finite non-abelian simple groups $S$ with $\omega(S) \leq n$. In Section \ref{genkohl} we generalize this result to the family of all finite groups without non-trivial abelian normal subgroups.

\ 

The results presented so far are complemented in Section \ref{o7} where we exhibit infinitely many finite non-solvable groups $G$ with $\omega(G)=7$. They are the semidirect products (with natural componentwise action) $\SL(2,\mathbb{F}_4) \ltimes V^m$ where $V = \mathbb{F}_4^2$ and $m \geq 2$ (the case $m=1$ is $\ASL(2,\mathbb{F}_4)$). Actually we prove the more general result $$\omega(\SL(2,F) \ltimes \mbox{M}(2 \times m,F)) = 3+\omega(\SL(2,F))$$ where $\mbox{M}(2 \times m,F)$ is the $F$-vector space consisting of all the $2 \times m$ matrices with coefficients in $F$, $m \geq 2$, $F$ is any field of characteristic $2$ with more than $2$ elements and $\SL(2,F)$ acts on $\mbox{M}(2 \times m,F)$ by multiplication from the left.

\section{Finite non-solvable groups with $\omega(G) \leq 6$} \label{o6}

In this section all considered groups are assumed to be finite. We solve Problem 2.5 of \cite{stroppel}. The following is the main result of this section.

\begin{teor} \label{main}
Let $G$ be a finite non-solvable group with $\omega(G) \leq 6$. Then $G$ is isomorphic to one of $\PSL(2,\mathbb{F}_4)$, $\PSL(2,\mathbb{F}_7)$, $\PSL(2,\mathbb{F}_8)$, $\PSL(2,\mathbb{F}_9)$, $\PSL(3,\mathbb{F}_4)$, $\ASL(2,\mathbb{F}_4)$. Moreover $\omega(\PSL(2,\mathbb{F}_4)) = 4$, $\omega(\PSL(2,\mathbb{F}_7)) = 5$, $\omega(\PSL(2,\mathbb{F}_8)) = 5$, $\omega(\PSL(2,\mathbb{F}_9)) = 5$, $\omega(\PSL(3,\mathbb{F}_4)) = 6$, $\omega(\ASL(2,\mathbb{F}_4)) = 6$.
\end{teor}

For this we need some preparatory results. A crucial observation, that we will use many times, is that two elements that lie in the same automorphism orbit have the same order; in particular $\omega(G) \geq m$, where $m$ is the number of element orders of $G$, $m = |\{o(g)\ :\ g \in G\}|$. Part of the following lemma is \cite[Lemma 3.2]{stroppel}.

\begin{lemma} \label{bound}
Let $G$ be a group and $N$ a characteristic subgroup of $G$. Then $\omega(G) \geq \omega(N) + \omega(G/N) - 1$. Moreover if equality holds then every non-trivial coset $xN$ is contained in an automorphism orbit of $G$.
\end{lemma}

\begin{proof}
Observe that $\Aut(G)$ acts naturally on $N$ and $G/N$ by automorphism. Let $a_1=1,a_2,\ldots,a_r$ be representatives of the $\Aut(G)$-orbits contained in $N$ and let $b_1=1,b_2,\ldots,b_s$ be elements of $G$ such that $b_1N,b_2N,\ldots,b_sN$ are representatives of the $\Aut(G)$-orbits of $G/N$. Observe that $r \geq \omega(N)$ and $s \geq \omega(G/N)$. It is clear that $a_1,\ldots,a_r,b_2,\ldots,b_s$ lie in distinct automorphism orbits of $G$, in particular $\omega(G) \geq r+s-1$. The stated inequality follows. Suppose equality holds, so that $\{a_1, a_2, \ldots, a_r, b_2, \ldots, b_s\}$ is a complete set of representatives of the automorphism orbits of $G$. If $x \in G-N$ there is an automorphism $\varphi$ of $G$ such that $x^{\varphi} = b_i$ for some $i \in \{2,\ldots,s\}$. Let $n \in N$, and let $\psi \in \Aut(G)$, $j \geq 2$ be such that $(xn)^{\psi} = b_j$ (such $j$ exists because $xn \not \in N$ being $x \not \in N$). We have $b_j^{\psi^{-1} \varphi} N = (xn)^{\varphi} N = x^{\varphi} n^{\varphi} N = x^{\varphi} N = b_i N$, hence $b_i=b_j$ therefore $x$ and $xn$ are $\Aut(G)$-conjugate.
\end{proof}

\begin{lemma} \label{kohl}
Let $S$ be a non-abelian simple group and let $G$ be an almost-simple group with $S \leq G \leq \Aut(S)$. If $\omega(G) \leq 6$ then $G=S$ and $S$ is isomorphic to $\PSL(2,\mathbb{F}_q)$ for $q \in \{4,7,8,9\}$ or to $\PSL(3,\mathbb{F}_4)$.
\end{lemma}

\begin{proof}
By Lemma \ref{bound} we have $\omega(S) \leq 6$ and by Kohl's classification \cite[Table 5]{kohl} we have that $S$ is isomorphic to one of the groups listed. Suppose $G \neq S$ by contradiction. By Lemma \ref{bound} we have $\omega(S) \leq 5$ so $S \cong \PSL(2,\mathbb{F}_q)$ with $q \in \{4,7,8,9\}$. If $q \neq 9$ then $S$ has prime index in $\Aut(S)$, and $G \cong A = \Aut(S)$ follows. However, we have $\omega(\Aut(S)) \in \{7,9,11\}$ in this case. There remains the case $q=9$; here $S = \PSL(2,\mathbb{F}_9) \cong \mbox{A}_6$, and $\Aut(S)/S$ is elementary abelian of order $4$. We have $\omega(\Aut(S))=13$. The intermediate groups are isomorphic to $\PGL(2,\mathbb{F}_9)$, $\mbox{S}_6$, and the (non-simple) Mathieu group $\mbox{M}_{10}$, with $\omega(\PGL(2,\mathbb{F}_9))=8$, $\omega(\mbox{S}_6)=8$, $\omega(\mbox{M}_{10})=7$. This proves that $G=S$.
\end{proof}

We now proceed to the proof of Theorem \ref{main}. Assume the result does not hold and let $G$ be a counterexample of minimal order. By hypothesis $G$ is non-solvable and $\omega(G) \leq 6$.

\subsection{} $Z(G)=\{1\}$.

\ 

Suppose $Z(G) \neq \{1\}$ by contradiction, and let $p$ be a prime divisor of $|Z(G)|$. Since $G$ is not solvable, $G/Z(G)$ is not solvable, so by Burnside's Theorem \cite[Theorem 3.3 of Chapter 4]{gorenstein} $|G/Z(G)|$ has at least three distinct prime divisors $p_1$, $p_2$, $p_3$. We choose $p_1=p$ if $p$ divides $|G/Z(G)|$. Now $G$ has elements of order $1$, $pp_2$, $pp_3$, $p_2$ and $p_3$. If $p_1 \neq p$ then we have elements of order $p$ and $p_1$, which exceeds $\omega(G) \leq 6$, so $p=p_1$ and $G$ has at least two automorphism orbits of $p$-elements of order at least $p$, one in $Z(G)$ and the other in $G-Z(G)$, a contradiction. Therefore $Z(G)=\{1\}$. \hfill $\qed$

\subsection{} \label{chm} Let $N$ be a minimal characteristic subgroup of $G$. Then $N$ is elementary abelian, $C_G(N)=N$ and $G/N \cong \PSL(2,\mathbb{F}_q)$ with $q \in \{4,7,8,9\}$.

\ 

If we already know that $N$ is elementary abelian then we deduce that $G/N$ is non-solvable, $\omega(G/N) \leq 5$ by Lemma \ref{bound} and the minimality of $G$ as a counterexample implies that $G/N$ is isomorphic to one of the groups listed. In particular $G/N$ is simple so $C_G(N)=N$. Since $N$ is characteristically simple, all we need to show is that $N$ is abelian. Assume $N$ is non-abelian by contradiction. If $N$ is not simple then $N$ contains non-abelian simple normal subgroups $N_1$ and $N_2$ that are isomorphic and which commute. By Burnside's Theorem \cite[Theorem 3.3 of Chapter 4]{gorenstein} $|N_1|$ has at least three prime divisors $p_1$, $p_2$ and $p_3$. Thus $G$ has elements of orders $1$, $p_1$, $p_2$, $p_3$, $p_1p_2$, $p_1p_3$, $p_2p_3$ so $\omega(G) \geq 7$, a contradiction. So $N$ is a non-abelian simple group. If $NC_G(N) \neq G$ then $\omega(NC_G(N)) \leq 5$ by Lemma \ref{bound} and so $C_G(N) = \{1\}$ by minimality of $G$ as a counterexample. But then $N \leq G \leq \Aut(N)$ and we obtain a contradiction via Lemma \ref{kohl}. Hence $G = NC_G(N) \cong N \times C_G(N)$ and $C_G(N) \neq \{1\}$. Now taking $x \in C_G(N)$ of prime order $p$, and noting that $|N|$ has at least three prime divisors $p_1$, $p_2$ and $p_3$ where we choose $p_1=p$ if $p$ divides $|N|$, we see that $G$ has elements of orders $p$, $pp_2$, $pp_3$ in $G-N$ and so $\omega(G) \geq 7$, a contradiction. \hfill $\qed$

\subsection{} $N$ is a $2$-group.

\ 

Assume $N$ is not a $2$-group by contradiction. Then being elementary abelian, $N$ is a $p$-group for some odd prime $p$. Let $R$ be a Sylow $2$-subgroup of $G$. Then by \ref{chm}, $R$ contains an elementary abelian subgroup $A$ of order $4$ (because $G/N$ does), so \cite[5.3.16]{gorenstein} implies $\langle C_N(x)\ :\ x \in A-\{1\} \rangle = N$, in particular there is $x \in A$ such that $xN$ contains elements of order $2p$, so Lemma \ref{bound} implies $6 \geq \omega(G) > \omega(N)+\omega(G/N)-1$ so $\omega(G/N)=4$, therefore $G/N \cong \mbox{A}_5$ by minimality of $G$ as a counterexample. Furthermore, we know that the automorphism orbits of $G$ are $\{1\}$, the non-trivial elements of $N$, the automorphism conjugates of non-trivial elements of $R$, $3$-elements which project to elements of order $3$ in $G/N$, $5$-elements which project to elements of order $5$ in $G/N$ and the aforementioned elements of order $2p$. Suppose that $p \neq 3$. Then there are no elements of order $3p$. Hence, if $x$ is an element of order $3$ then $C_N(x)=\{1\}$. Choose $x$ of order $3$ such that $xN$ normalizes $RN/N$. We have $N_G(RN)/N = N_{G/N}(RN/N) = XQ \cong \mbox{A}_4$ where $X=\langle x \rangle N/N$ and $Q = RN/N$, so since $C_G(N)=N$ \cite[36.2]{aschbacher} provides a contradiction. Therefore $p=3$. Let $T$ be a Sylow $3$-subgroup of $G$. Then, as all the non-trivial elements of $N$ are in the same $\Aut(G)$-orbit, and $N$ is a normal subgroup of $T$, every element of $N$ is centralized by a Sylow $3$-subgroup of $G$ (because the center of $T$ intersects $N$ non-trivially). Writing $|N|=3^n$, since $G$ has $10$ Sylow $3$-subgroups we then have $10(|C_N(T)|-1) \geq 3^n-1$. As $G/N \cong \mbox{A}_5$ can be generated by two Sylow $3$-subgroups, $C_N(T) \cap C_N(T^g) \subseteq Z(G) = \{1\}$ for $T^g \neq T$ so $|C_N(T)| \leq 3^{n/2}$. Hence $10(3^{n/2}-1) \geq 3^n-1$ so $n \leq 4$. Since $C_G(N)=N$ we have $G/N \leq \Aut(N) = \GL(n,\mathbb{F}_3)$ and since $5$ divides $|G/N|$ we deduce at once that $n=4$ and $N$ is a minimal normal subgroup of $G$, giving an irreducible action of $A_5$ on ${\mathbb{F}_3}^4$. There is only one such action, the natural coordinate permutation action on the fully deleted module $N=\{v \in {\mathbb{F}_3}^5\ :\ \sum_{i=1}^5 v_i = 0\}$. But then $\bigcup_{g \in G} C_N(T^g)$ does not equal $N$, since there are vectors of $N$ not centralized by $3$-cycles, for example $(1,2,1,2,0)$. This is a contradiction. \hfill $\qed$

\subsection{} $G/N \cong \mbox{A}_5$ and $G$ has elements of order $1,2$ in $N$ and $2,3,4,5$ outside $N$.

\ 

We know that $G/N \cong \PSL(2,\mathbb{F}_q)$ with $q \in \{4,7,8,9\}$. There are at least two $\Aut(G)$-orbits contained in $N$ and there are at least three $\Aut(G)$-orbits outside $N$ (at least four if $q \neq 4$), of the form $xN$ with $x$ of prime power order. If $q \neq 4$ then $\omega(G)=6$, in particular equality holds in Lemma \ref{bound} so since the cosets of $N$ are of the form $xN$ with $x$ of prime power order, $G$ does not have elements of order $2r$ for any odd prime $r$ dividing $|G|$; let $r=3$ if $q=7$, $r=7$ if $q=8$ and $r=5$ if $q=9$. If $q=4$ then since $\omega(G) \leq 6$, $G$ cannot have elements of order $6$ and elements of order $10$, let $r \in \{3,5\}$ be such that $G$ does not have elements of order $2r$ in this case. Let $R$ be a Sylow $r$-subgroup of $G$ and $H := N_G(R)$. Note that $H \cap N = \{1\}$, indeed $R$ and $H \cap N$ normalize each other and have trivial intersection, so if $H \cap N \neq \{1\}$ then $R(H \cap N)$ contains elements of order $2r$, a contradiction. We deduce $H \cong H/H \cap N \cong HN/N$. Since $(|R|,|N|)=1$, a Frattini-type argument implies $N_G(RN)=HN$, and by the choice of $r$ we have that $N_{G/N}(RN/N)$ is a dihedral group of order $2r$. Therefore $H \cong N_G(RN)/N = N_{G/N}(RN/N)$ is a dihedral group, so since $H \cap N = \{1\}$, there exists an involution $x \in H-N$. Since $x$ does not centralize $N$, there exists $y \in N$ with $xy \neq yx$ hence $xy$ has order $4$. In particular $xN$ contains elements of orders $2$ and $4$ so Lemma \ref{bound} implies $6 \geq \omega(G) > \omega(N)+\omega(G/N)-1$ hence $\omega(G/N)=4$ so $G/N \cong A_5$ by minimality of $G$ as a counterexample. \hfill $\qed$

\ 

Now let $T$ be a Sylow $2$-subgroup of $G$. Then $G=\langle T,T^g \rangle$ for $g \in G-N_G(T)$. Set $|N|=2^n$. Then $C_N(T) \cap C_N(T^g) \subseteq Z(G) = \{1\}$ hence $|C_N(T)| \leq 2^{n/2}$. On the other hand, as all the non-trivial elements of $N$ are $\Aut(G)$-conjugate and $N \unlhd T$, every element of $N$ is centralized by a Sylow $2$-subgroup of $G$ (because the center of $T$ intersects $N$ non-trivially). Since $G$ has five Sylow $2$-subgroups, we have $5(|C_N(T)|-1) \geq 2^n-1$ hence $n \leq 4$. Since $C_G(N)=N$, $G/N \leq \Aut(N) = \GL(n,\mathbb{F}_2)$ so since $5$ divides $|G/N|$ we have $n=4$, $|G|=960$ and $N$ is a minimal normal subgroup of $G$. In particular $G$ is perfect. By \cite{perfect} there are only two perfect groups of order $960$, one is $\ASL(2,\mathbb{F}_4)$, the other has elements of order $6$. This leads to the final contradiction. The proof of Theorem \ref{main} is completed.

\section{Groups without non-trivial abelian normal subgroups} \label{genkohl}

This section gives a generalization of the main result of Kohl's paper \cite{kohl}. As in that paper, we need the classification of the finite simple groups in our proof.

\begin{teor} \label{noabf}
Let $n$ be a natural number. Then there are only finitely many finite groups $G$ without non-trivial abelian normal subgroups such that $\omega(G) \leq n$.
\end{teor}

\begin{proof}
Let $G$ be a finite group without non-trivial abelian normal subgroups, such that $\omega(G) \leq n$, and let $N := \mbox{soc}(G)$, the socle of $G$, that is, the subgroup of $G$ generated by the minimal normal subgroups of $G$. Suppose that the centralizer $C_G(N)$ (which is a normal subgroup of $G$) is non-trivial and let $M$ be a minimal normal subgroup of $G$ contained in $C_G(N)$. By definition of $N$ we have $M \subseteq N$ hence $M$ centralizes itself, implying that $M$ is abelian, contradicting the fact that $G$ does not have non-trivial abelian normal subgroups. This shows that $C_G(N) = \{1\}$, hence the conjugation action gives an injective homomorphism $G \to \Aut(N)$, implying that $|G| \leq |\Aut(N)| \leq |N|^{|N|}$. This means that it is enough to show that there are only a finite number of possibilities for $N$.

\ 

$N$ is a direct product of some minimal normal subgroups, write $N = S_1^{m_1} \times \cdots \times S_t^{m_t}$, where each $S_i$ is a non-abelian simple group and $S_i \not \cong S_j$ if $i \neq j$. Observe that each $S_i^{m_i}$ is characteristic in $N$, hence $t \leq \omega(N)$ and $\omega(S_i^{m_i}) \leq \omega(N)$. Moreover $\omega(S_i^{m_i}) = \binom{m_i+\omega(S_i)-1}{\omega(S_i)-1}$ by \cite[3.17]{stroppel}, so $m_i \leq \omega(S_i^{m_i})$, $\omega(S_i) \leq \omega(S_i^{m_i})$. Since $\omega(N) \leq \omega(G) \leq n$ being $N$ characteristic in $G$, we deduce $m_i \leq n$, $\omega(S_i) \leq n$ for all $i$ and $t \leq n$. By \cite{kohl}, there are only finitely many possibilities for each $S_i$ hence also for $N$.
\end{proof}

\section{Infinitely many non-solvable groups with $\omega(G)=7$} \label{o7}

Let $F$ be a field and $m$ a positive integer. Consider the group $G := G_{m,F} := \SL(2,F) \ltimes V$ where $\SL(2,F)$ (the group of $2 \times 2$ matrices of determinant $1$ with coefficients in $F$) acts on $V := \mbox{M}(2 \times m,F)$ (the $F$-vector space of $2 \times m$ matrices with coefficients in $F$) by multiplication from the left. Observe that $G_{1,\mathbb{F}_4} = \ASL(2,\mathbb{F}_4)$. In the following theorem we compute $\omega(G)$ when $F$ has characteristic $2$.

\begin{teor} \label{inf7}
Let $F$ be a (possibly infinite) field of characteristic $2$ with more than $2$ elements. If $m \geq 2$ then $\omega(G_{m,F}) = 3+\omega(\SL(2,F))$.
\end{teor}

\begin{proof}
We identify $G$ with $\left\{ \left( \begin{array}{cc} X & Y \\ 0 & 1 \end{array} \right)\ :\ X \in \SL(2,F),\ Y \in \mbox{M}(2 \times m,F) \right \}$. \\ $G$ is a normal subgroup of $$\Gamma := \left\{ \left( \begin{array}{cc} A & B \\ 0 & C \end{array} \right)\ :\ A \in \GL(2,F),\ B \in \mbox{M}(2 \times m,F),\ C \in \GL(m,F) \right\}.$$Observe that $\Gamma$ acts on $G$ by conjugation hence every automorphism orbit of $G$ is a union of $\Gamma$-orbits. In what follows we make use of the identity $$\left( \begin{array}{cc} A & B \\ 0 & C \end{array} \right)^{-1} \left( \begin{array}{cc} X & Y \\ 0 & 1 \end{array} \right) \left( \begin{array}{cc} A & B \\ 0 & C \end{array} \right) = \left( \begin{array}{cc} A^{-1}XA & A^{-1}(X-1)B+A^{-1}YC \\ 0 & 1 \end{array} \right).$$

We identify $V$ with $M = \left\{ \left( \begin{array}{cc} 1 & Y \\ 0 & 1 \end{array} \right)\ :\ Y \in \mbox{M}(2 \times m,F) \right \}$. Observe that $M$ is the unique maximal abelian normal subgroup of $G$. Indeed, if $K$ is any abelian normal subgroup then $KM/M \cong K/K \cap M$ is an abelian normal subgroup of $G/M \cong \SL(2,F)$ which is simple since $|F| \geq 4$ (note that $\SL(2,F) \cong \PSL(2,F)$ being $F$ of characteristic $2$). Thus $KM/M$ is trivial, in other words $K \leq M$. In particular $M$ is characteristic in $G$.

\ 

We claim that $M$ contains at least three $\Aut(G)$-orbits. For $v,w \in \mbox{M}(1 \times m,F)$ set $T_{v,w} := \left( \begin{array}{c} v \\ w \end{array} \right)$ and $g_{v,w} := \left( \begin{array}{cc} 1 & T_{v,w} \\ 0 & 1 \end{array} \right) \in M$. Observe that $\alpha = \left( \begin{array}{cc} X & Y \\ 0 & 1 \end{array} \right) \in G$ centralizes $g_{v,w}$ if and only if $(X-1)T_{v,w}=0$. Suppose $v$ and $w$ are linearly independent. Then we easily deduce $C_G(g_{v,w})=M$. Let $J := \left( \begin{array}{cc} 1 & 1 \\ 0 & 1 \end{array} \right)$. Both $g_{v,w}$ and $h = \left( \begin{array}{cc} J & T_{v,0} \\ 0 & 1 \end{array} \right)$ commute with $g_{v,0}$, but $g_{v,w}h \neq hg_{v,w}$. This implies that $C_G(g_{v,0})$ is not abelian. Hence $1$, $g_{v,w}$, $g_{v,0}$ have non-isomorphic centralizers, in particular they lie in different $\Aut(G)$-orbits.

\ 

We claim that $M$ contains exactly three $\Aut(G)$-orbits. By the above, it is enough to prove that $M$ contains at most three $\Aut(G)$-orbits. Let $\left( \begin{array}{cc} 1 & Y \\ 0 & 1 \end{array} \right) \in M$, $e_1 = (1,0,\ldots,0)$ and $e_2 = (0,1,0,\ldots,0)$ in $\mbox{M}(1 \times m,F)$. There exists $C \in \GL(m,F)$ such that $YC = \left( \begin{array}{c} e_1 \\ 0 \end{array} \right)$ or $YC = \left( \begin{array}{c} e_1 \\ e_2 \end{array} \right)$ according to whether $Y$ has rank $1$ or $2$. Hence $M$ contains at most three $\Gamma$-orbits, so $M$ contains at most three $\Aut(G)$-orbits.

\ 

We claim that any non-trivial element of $\SL(2,F)$ of order a power of $2$ has order $2$, and the elements of order $2$ in $\SL(2,F)$ are all conjugated to $J := \left( \begin{array}{cc} 1 & 1 \\ 0 & 1 \end{array} \right)$ in $\GL(2,F)$. Consider $X \in \SL(2,F)$ of order $2^s$ with $s \geq 1$. Observe that $X^{2^s}=1$ implies that the eigenvalues $\lambda$ of $X$ verify $\lambda^{2^s}=1$ so being $F$ of characteristic $2$, $(\lambda-1)^{2^s}=0$ thus $\lambda = 1$. Since $s \geq 1$, $X$ is conjugated to $J$ in $\GL(2,F)$. In particular $s=1$.

\ 

We claim that $\omega(G) \geq 3+\omega(G/M)$. $M$ is an elementary abelian $2$-group, and $\left( \begin{array}{cc} J & \left( \begin{array}{c} 0 \\ e_1 \end{array} \right) \\ 0 & 1 \end{array} \right)^2 = \left( \begin{array}{cc} 1 & \left( \begin{array}{c} e_1 \\ 0 \end{array} \right) \\ 0 & 1 \end{array} \right)$ has order $2$. Thus there are elements of order $4$ in $G$, and by the discussion above, in $M$ there are two $\Aut(G)$-orbits of elements of order $2$. Since $\SL(2,F)$ has elements of order $2$ and no elements of order $4$, by adapting Lemma \ref{bound} we obtain $\omega(G) \geq 3+\omega(G/M)$.

\ 

Each involution $g = \left( \begin{array}{cc} X & Y \\ 0 & 1 \end{array} \right) \in G-M$ is conjugated to $\left( \begin{array}{cc} J & 0 \\ 0 & 1 \end{array} \right)$ in $\Gamma$. Indeed, since $X$ is conjugated to $J$ in $\GL(2,F)$, we may assume $X=J$. Then $g^2=1$ implies $Y = \left( \begin{array}{c} v \\ 0 \end{array} \right)$ with $v \in \mbox{M}(1 \times m,F)$. Setting $Z := \left( \begin{array}{c} 0 \\ v \end{array} \right)$ we obtain $\left( \begin{array}{cc} 1 & Z \\ 0 & 1 \end{array} \right)^{-1} \left( \begin{array}{cc} J & Y \\ 0 & 1 \end{array} \right) \left( \begin{array}{cc} 1 & Z \\ 0 & 1 \end{array} \right) = \left( \begin{array}{cc} J & 0 \\ 0 & 1 \end{array} \right)$.

\ 

We claim that each element $g = \left( \begin{array}{cc} X & Y \\ 0 & 1 \end{array} \right)$ of order $4$ is conjugated to the element $\left( \begin{array}{cc} J & \left( \begin{array}{c} 0 \\ e_1 \end{array} \right) \\ 0 & 1 \end{array} \right)$ in $\Gamma$. Again, we may assume $X=J$, but $Y = \left( \begin{array}{c} v \\ w \end{array} \right)$ with $w \neq 0$ now. Using $C \in \GL(m,F)$ with $wC=e_1$ we see that $g$ is conjugated to $\left( \begin{array}{cc} J & Y_1 \\ 0 & 1 \end{array} \right)$ with $Y_1 = \left( \begin{array}{c} x \\ e_1 \end{array} \right)$ for some $x \in \mbox{M}(1 \times m,F)$. Now $Z := \left( \begin{array}{c} 0 \\ x \end{array} \right)$ yields $\left( \begin{array}{cc} 1 & Z \\ 0 & 1 \end{array} \right)^{-1} \left( \begin{array}{cc} J & Y_1 \\ 0 & 1 \end{array} \right) \left( \begin{array}{cc} 1 & Z \\ 0 & 1 \end{array} \right) = \left( \begin{array}{cc} J & \left( \begin{array}{c} 0 \\ e_1 \end{array} \right) \\ 0 & 1 \end{array} \right).$

\ 

We claim that if $g = \left( \begin{array}{cc} X & Y \\ 0 & 1 \end{array} \right) \in G$ with $X^2 \neq 1$ (in other words, the order of $g$ is not $1$, $2$ or $4$) then $g$ is conjugated to $\left( \begin{array}{cc} X & 0 \\ 0 & 1 \end{array} \right)$ in $G$. To prove this observe that $X^2 \neq 1$ implies that $X-1$ is invertible, because if not then $\ker(X-1)$ is non-trivial, hence $X-1$ is conjugated to a matrix of the form $\left( \begin{array}{cc} 0 & \alpha \\ 0 & \beta \end{array} \right)$, thus $X$ is conjugated to $\left( \begin{array}{cc} 1 & \alpha \\ 0 & 1+\beta \end{array} \right)$, and $\det(X)=1$ implies $\beta=0$ hence $X^2=1$, a contradiction. Set $Z := (1-X)^{-1}Y$. Then $\left( \begin{array}{cc} 1 & Z \\ 0 & 1 \end{array} \right)^{-1} \left( \begin{array}{cc} X & Y \\ 0 & 1 \end{array} \right) \left( \begin{array}{cc} 1 & Z \\ 0 & 1 \end{array} \right) = \left( \begin{array}{cc} X & 0 \\ 0 & 1 \end{array} \right)$.

\ 

Now being $F$ a field of characteristic $2$, the automorphism group of $\SL(2,F)$ is generated by the automorphisms induced by conjugation in $\GL(2,F)$ and the field automorphisms (see for example \cite[Page 756]{hua}). Observe that using conjugation in $\Gamma$ and field automorphism action on each matrix entry of elements of $G$ we see that the set of elements listed so far includes a set of representatives of the automorphism orbits of $G$. Therefore $\omega(G) \leq 3+\omega(\SL(2,F))$ and the result follows.
\end{proof}
Since $\omega(\SL(2,\mathbb{F}_4))=4$ we obtain that $\omega(G_{m,\mathbb{F}_4}) = 7$ for all $m \geq 2$. In particular there are infinitely many finite non-solvable groups $G$ with $\omega(G)=7$.

\section{Acknowledgements}
We are really grateful to the referee and the editors for very carefully reading a previous version of this paper and for the insightful suggestions that contributed greatly to improve and simplify the proofs of Theorems \ref{main}, \ref{noabf}, \ref{inf7}.


\begin{thebibliography}{20}
\bibitem{aschbacher} M. Aschbacher, Finite Group Theory. Cambridge Studies in Advanced Mathematics, 10. Cambridge University Press, Cambridge, 2000.
\bibitem{gorenstein} D. Gorenstein; Finite Groups. Second edition. Chelsea Publishing Co., New York, 1980.
\bibitem{perfect} D. Holt, W. Plesken; Perfect groups. With an appendix by W. Hanrath. Oxford Mathematical Monographs. Oxford Science Publications. The Clarendon Press, Oxford University Press, New York, 1989.
\bibitem{kohl} S. Kohl; Classifying Finite Simple Groups with Respect to the Number of Orbits Under the Action of the Automorphism Group. Comm. Algebra 32 (2004), no. 12, 4785--4794. 
\bibitem{lm} T. J. Laffey, D. MacHale; Automorphism Orbits of Finite Groups. J. Austral. Math. Soc. Ser. A 40 (1986), no. 2, 253--260.
\bibitem{landau} E. Landau, \"Uber die Klassenzahl der binaren quadratischen Formen von negativer Diskriminant, Math. Ann. 56 (1903) 671--676. 
\bibitem{hua} Loo-Keng Hua, On the automorphisms of the symplectic group over any field. Ann. of Math. (2) 49, (1948). 739--759.
\bibitem{stroppel} M. Stroppel; Locally compact groups with many automorphisms. J. Group Theory 4 (2001), no. 4, 427--455. 
\bibitem{stroppel2} M. Stroppel; Locally compact groups with few orbits under automorphisms. Topology Proc. 26 (2001/02), no. 2, 819--842. 
\bibitem{zhang} J. Zhang; On Finite Groups All of Whose Elements of the Same Order Are Conjugate in Their Automorphism Groups. J. Algebra 153 (1992), no. 1, 22--36. 
\end{thebibliography}
\end{document}